\documentclass[11pt,a4paper]{article}
\usepackage[T1]{fontenc}
\usepackage{lmodern}
\usepackage[a4paper,margin=26mm,headheight=14pt]{geometry}
\usepackage{amsmath,amssymb,amsthm,mathtools}
\usepackage{microtype}
\usepackage{xcolor}
\usepackage{enumitem}
\usepackage{hyperref}
\hypersetup{colorlinks=true,linkcolor=red,anchorcolor=blue,citecolor=blue}
\numberwithin{equation}{section}
\newtheorem{theorem}{Theorem}[section]
\newtheorem{lemma}[theorem]{Lemma}
\theoremstyle{remark}
\newtheorem{remark}[theorem]{Remark}
\newcommand{\Pp}{\mathbb P}
\newcommand{\E}{\mathbb E}
\newcommand{\R}{\mathbb R}
\newcommand{\Z}{\mathbb Z}
\newcommand{\capn}{\operatorname{cap}}
\newcommand{\LE}{\operatorname{LE}}
\newcommand{\ind}{\mathbf 1}

\title{\vspace{-10mm}\Large\bfseries An explicit lower bound for the growth exponent of three-dimensional loop-erased random walk}
\author{Runsheng Liu\thanks{School of Mathematical Sciences, Peking University. \href{mailto:liurunsheng@pku.edu.cn}{liurunsheng@pku.edu.cn}}}

\begin{document}
\maketitle
\vspace{-7mm}
\begin{abstract}
In this paper, we derive a new lower bound for the Hausdorff dimension of 3D Brownian cut points by proving an explicit upper bound ($<0.9999$) for $\xi_3(1,1)$, the intersection exponent for two independent Brownian motions in 3D. Consequently, the growth exponent of 3D loop-erased random walk is at least $1.0001$.
\end{abstract}

\section{Introduction}

The loop-erased random walk (LERW), first introduced in Lawler \cite{Law80}, is defined by erasing the loops formed by random walk paths chronologically. It has been widely studied in the last forty years. In four dimensions and above, its scaling limit is Brownian motion; see \cite{Law80,Law86}. In two dimensions, its scaling limit is proven to be SLE$_2$ by Schramm \cite{Sch00}. For the three-dimensional case, Kozma \cite{Koz07} shows the existence of a scaling limit. Shiraishi \cite{Shi18} shows the existence of a growth exponent $\beta$, which also characterizes the Hausdorff dimension of the scaling limit. The explicit value of $\beta$ is a priori unknown, and the best bound before was $\beta\in(1,5/3]$. Simulations in \cite{Wil10} suggests that $\beta\approx 1.624$.

Our main result is the following.
\begin{theorem}\label{thm:main}
Define
\begin{equation}\label{eq:delta}
 \delta_0=\frac{-\log(1-1/1800)}{4\log4}.
\end{equation}
Then we have
\[
 \xi_3(1,1)\le1-\delta_0,
 \qquad \beta\ge1+\delta_0>1.0001,
\]
where $\xi_3(1,1)$ is the intersection exponent for two independent 3D Brownian motions and $\beta$ is the growth exponent for 3D LERW (their exact definitions will be recalled below).
\end{theorem}

We first recall the definition of the growth exponent of 3D LERW. Throughout, let $B(0,r)$ be the open Euclidean ball of radius $r$. Let $S$ be nearest-neighbour simple random walk on $\Z^3$, started at the origin. Set
\[T_R:=\inf\{n\ge0:|S_n|\ge R\},\]
and
\[M_R:=\operatorname{len}\bigl(\LE(S[0,T_R])\bigr),\]
where length means the number of edges and $\LE(\cdot)$ denotes the loop-erasure of the path. The growth exponent is defined by
\begin{equation}\label{eq:beta}
\beta:=\lim_{R\to\infty}\frac{\log\E M_R}{\log R}.
\end{equation}
The existence of this limit is given by \cite[Theorem 1.1.4]{Shi18}. There is also an equivalent definition of $\beta$ in \cite{Shi18}. Let $S'$ be another independent simple random walk started from $0$ and $T'_R$ be its stopping time respectively. Then there exists $\alpha>0$, such that
\[\Pp\big(S'(0,T'_R]\cap\LE(S[0,T_R])=\varnothing\big)=R^{-\alpha+o(1)},\]
where $\alpha=2-\beta$ and $o(1)$ tends to $0$ as $R\to\infty$.

We then briefly recall the definition of intersection exponents for Brownian motions and refer to \cite[Section 2.4]{GLLQ} for a detailed overview. Let $W,X$ be two independent Brownian motions started uniformly from $\partial B(0,1)$ in $\R^3$. Write $\sigma_r$ and $\tau_r$ for their respective first hitting times of $\partial B(0,r)$, with $\sigma_1=\tau_1=0$. For $R>1$, put
\begin{equation}\label{eq:q}
 Q_R(W)=\Pp\bigl(X[0,\tau_R]\cap W[0,\sigma_R]=\varnothing\mid W\bigr).
\end{equation}
Then, the intersection exponent $\xi_3(1,\lambda)$ is defined by
\[\E[Q_R^{\lambda}]=R^{-\xi_{3}(1,\lambda)+o(1)},\]
where the $o(1)$ term tends to zero as $R\to\infty$. Lawler \cite{Law98} provides the following stronger estimates.
\begin{equation}\label{eq:moments}
 \E Q_R\asymp R^{-\xi_3(1,1)},
 \qquad \E Q_R^2\asymp R^{-1}.
\end{equation}
Here $\asymp$ means comparison above and below by positive constants independent of $R$. The second assertion is the exact identity $\xi_3(1,2)=1$; see \cite{Law89,BL1990}. The Hausdorff dimension of Brownian cut points is proven to be $2-\xi_3(1,1)$ in \cite{Law96}. Since LERW is a subset of random walk, a trivial inequality yields that $\alpha:=2-\beta\leq\xi_{3}(1,1)$. Hence
\begin{equation}\label{eq:cut}
\beta\ge2-\xi.
\end{equation}

% Our main result is the following.
% \begin{theorem}\label{thm:main}
% Define
% \begin{equation}\label{eq:delta}
%  \delta_0=\frac{-\log(1-1/1800)}{4\log4}.
% \end{equation}
% Then we have
% \[
%  \xi_3(1,1)\le1-\delta_0,
%  \qquad \beta\ge1+\delta_0>1.0001.
% \]
% \end{theorem}

We now briefly introduce the proof and the structure of the paper. In Section 2, we provide a uniform control for a Brownian motion to avoid a set with given capacity. In Section 3, we show that it is quite unlikely for a Brownian segment between $\partial B(0,r)$ and $\partial B(0,2r)$ to have small capacity. In Section 4, we combine these two estimates to conclude the proof of Theorem~\ref{thm:main}.

\begin{remark}
Using the same strategy, we can prove a slightly better upper bound for $\xi_3(1,1)$, namely $\xi_3(1,1)<0.999$. Since this improvement brings more technical issues and is still very far from optimal, we do not include it in the paper.
\end{remark}

\medskip

\noindent {\bf Acknowledgements:} This work is supported by the National Key R\&D Program of China (No. 2021YFA1002700) and Beijing Natural Science Foundation (JQ26001). We obtain the results with the assistance from GPT-6 Astra. We thank Xinyi Li and Tian Yu for useful discussions.
\section{Capacity and a uniform hitting estimate}\label{sec:cap}

For a compact set $K\subset\R^3$, we use Newtonian capacity normalized by
\begin{equation}\label{eq:capacity}
 \capn(K):=\left[\inf_{\mu\in\mathcal P(K)}I(\mu)\right]^{-1},
\end{equation}
where $\mathcal{P}(K)$ consists of all the probability measures supported on $K$ and
\[I(\mu):=\iint\frac{\mu(dx)\mu(dy)}{|x-y|}.\]
As usual $1/\infty=0$. Under this normalization, a closed ball of radius $a$ has capacity $a$. For a Brownian motion $X$ started at $x$, define
\[
 H_K:=\inf\{t\ge0:X_t\in K\},
 \qquad \tau_D:=\inf\{t\ge0:X_t\notin D\}.
\]
In particular, hitting at time zero is allowed.

The following lemma provides a uniform lower bound for a Brownian motion to avoid a set with given capacity.

\begin{lemma}\label{lem:hitting}
For every $r>0$, compact $K\subseteq\overline B(0,2r)$, and $|x|\le r$,
\begin{equation}\label{eq:hitting}
 \Pp_x(H_K<\tau_{B(0,4r)})\ge\frac{\capn(K)}{9r}.
\end{equation}
\end{lemma}

\begin{proof}
Let $D=B(0,4r)$. Normalize the Green's function $g_D$ so that its singularity is $|x-y|^{-1}$; thus $g_D$ is $2\pi$ times the occupation Green density for Brownian motion with generator $\frac12\Delta$. The explicit form of Green's function in a ball is given by
\begin{equation}\label{eq:green}
 g_D(x,y)=\frac1d-\frac1{\sqrt{d^2+A}},
\end{equation}
where
\[d=|x-y|,
 \quad A=\frac{(16r^2-|x|^2)(16r^2-|y|^2)}{16r^2}.\]
The formula at $y=0$ follows by continuity, and the value at $x=y$ is understood as $+\infty$. For $x,y$ such that $|x|\le r$ and $|y|\le 2r$, we have
\[
 d\le3r,\qquad A\ge\frac{45}{4}r^2.
\]
For $d,A>0$, the function $d^{-1}-(d^2+A)^{-1/2}$ decreases with $d$ and increases with $A$. Consequently
\begin{equation}\label{eq:green-min}
 g_D(x,y)\ge\frac1{3r}-\frac1{(9/2)r}=\frac1{9r}.
\end{equation}

For completeness, we define the following relative capacity
\[
 \capn_D(K)=\left[\inf_{\mu\in\mathcal P(K)}
 \iint g_D(u,v)\,\mu(du)\mu(dv)\right]^{-1}.
\]
Since $g_D(u,v)\le|u-v|^{-1}$, we have $\capn_D(K)\ge\capn(K)$. If $\capn(K)>0$, let $e^D_K$ be the relative equilibrium measure, of total mass $\capn_D(K)$. By the standard equilibrium-potential identity (see e.g.~\cite[Theorem 7.28]{BH15}), for any $x\in D\setminus K$, we have
\[
 \Pp_x(H_K<\tau_D)=\int_K g_D(x,y)\,e^D_K(dy).
\]
Applying~\eqref{eq:green-min} gives~\eqref{eq:hitting}. If $\capn(K)=0$, the desired lower bound is immediate. If $x\in K$, the left-hand side is $1$, while $\capn(K)\le2r$ by monotonicity. Combining all cases completes the proof. %Only the identity off $K$ is used, so the possible exceptional set in the equilibrium-potential identity on $K$ introduces no assumption about regularity of $K$.
\end{proof}

\section{An explicit capacity estimate for one Brownian segment}\label{sec:estimate}

Let $(\mathcal F_t)$ be the usual Brownian filtration of $W$, including its initial position. For $r\ge1$, set
\[
 K(r):=W[\sigma_r,\sigma_{2r}]
\]
be the segment from first hitting $\partial B(0,r)$ to $\partial B(0,2r)$ with respect to $W$, which is a compact subset of $\overline B(0,2r)$.

\begin{lemma}\label{lem:capacity}
For every $r\ge1$, almost surely,
\begin{equation}\label{eq:bad-capacity}
 \Pp\left(\capn(K(r))<\frac r{200}\,\middle|\,\mathcal F_{\sigma_r}\right)<\frac1{18}.
\end{equation}
\end{lemma}

\begin{proof}
Condition on $\mathcal F_{\sigma_r}$. By the strong Markov property,
\[
 V_s=W_{\sigma_r+s}-W_{\sigma_r},\qquad s\ge0,
\]
is a standard Brownian motion started at zero, independent of the conditioned past. Put $t=r^2/25$. Since $|W_{\sigma_r}|=r$, an exit from $B(0,2r)$ before time $\sigma_r+t$ requires $\sup_{s\le t}|V_s|\ge r$.

The process $Z_s:=\exp(10|V_s|^2/r^2)$, $0\le s\le t$, is a non-negative integrable submartingale. Convexity gives the submartingale property, and $20s/r^2\le4/5$ gives integrability. Doob's maximal inequality and the Gaussian exponential-moment formula imply
\begin{align}
 \Pp\left(\sigma_{2r}-\sigma_r<t\,\middle|\,\mathcal F_{\sigma_r}\right)
 &\le\Pp\left(\sup_{s\le t}|V_s|\ge r\right)\notag\\
 &\le e^{-10}\E Z_t
 =e^{-10}(1-20t/r^2)^{-3/2}
 =5^{3/2}e^{-10}. \label{eq:confinement}
\end{align}

Now consider the occupation probability measure
\[
 \mu_t:=\frac1t\int_0^t\delta_{W_{\sigma_r+s}}\,ds.
\]
For $s\ne u$, the increment $V_s-V_u$ is centered Gaussian with covariance $|s-u|I_3$, so
\[
 \E|V_s-V_u|^{-1}=\sqrt{\frac2\pi}\,|s-u|^{-1/2}.
\]
Integration by part then gives
\begin{align}
 \E I(\mu_t)
 &=\frac1{t^2}\sqrt{\frac2\pi}
   \int_0^t\!\int_0^t |s-u|^{-1/2}\,ds\,du\notag\\
 &=\frac83\sqrt{\frac2\pi}\,t^{-1/2}
 =\frac{40}{3r}\sqrt{\frac2\pi}. \label{eq:energy}
\end{align}
The diagonal $s=u$ has two-dimensional Lebesgue measure zero, and the displayed integral is finite.

On the event $\{\sigma_{2r}-\sigma_r\ge t\}$, the measure $\mu_t$ is supported on $K(r)$. Therefore, on that event, $\capn(K(r))<r/200$ implies $I(\mu_t)>200/r$. A union bound,~\eqref{eq:confinement}, and Markov's inequality yield
\[
 \Pp\left(\capn(K(r))<\frac r{200}\,\middle|\,\mathcal F_{\sigma_r}\right)
 \le5^{3/2}e^{-10}+\frac r{200}\E I(\mu_t)=p_0.
\]
Finally, $5^{3/2}e^{-10}<1/1000$ and $\sqrt{2/\pi}<4/5$, combining with \eqref{eq:energy} gives
\[
 p_0<\frac1{1000}+\frac4{75}=\frac{163}{3000}<\frac1{18},
\]
which concludes the proof.
\end{proof}

\section{Proof of the main theorem}
In this section, we complete the proof of Theorem~\ref{thm:main}.

Fix an integer $N\ge1$, set $R=4^N$, and define $r_j:=4^j$ for $0\le j\le N$, so $r_N=R$. For $0\le j<N$, we set
\[K_j=K(r_j):=W[\sigma_{r_j},\sigma_{2r_j}].\]
We say the $j$-th layer is good, denoted by $G_j$, if $\capn(K_j)\ge r_j/200$. Let
\[L_N=\sum_{j=0}^{N-1}\ind_{G_j}\]
denote the total number of good layers for $0\le j\le n$. Since $2r_j<R$, every $K_j$ is a compact subset of $W[0,\tau_R]$. We also note that the event $G_j$ is measurable with respect to $\mathcal F_{\sigma_{2r_j}}$, and
\begin{equation}\label{eq:adapted}
 \mathcal F_{\sigma_{2r_j}}\subseteq\mathcal F_{\sigma_{r_{j+1}}}.
\end{equation}
The following lemma provides an upper tail estimate for the number of good layers.
\begin{lemma}\label{lem:count}
With
\begin{equation}\label{eq:b}
 b=\frac{\frac34\log18-\log2}{\log4}>1,
\end{equation}
we have $\Pp(L_N<N/4)\le R^{-b}$.
\end{lemma}

\begin{proof}
Let $J=\{j_1<\cdots<j_m\}\subseteq\{0,\ldots,N-1\}$. The preceding bad events are measurable at $\sigma_{r_{j_m}}$ by~\eqref{eq:adapted}. Lemma~\ref{lem:capacity} and the tower property give
\[
 \Pp\left(\bigcap_{j\in J}G_j^c\right)
 \le\frac1{18}\Pp\left(\bigcap_{\ell=1}^{m-1}G_{j_\ell}^c\right)
 \le\cdots
 \le18^{-m}.
\]
If $L_N<N/4$, more than $3N/4$ indices are bad. Such a configuration contains a bad index set of size $k=\lceil3N/4\rceil$. Applying union bound, we conclude that
\begin{equation}\label{eq:count}
 \Pp(L_N<N/4)
 \le\binom Nk18^{-k}
 \le2^N18^{-3N/4}=R^{-b}.
\end{equation}
Finally, $b>1$ is equivalent to $18^3>8^4$, namely $5832>4096$, which completes the proof.
\end{proof}
The next lemma confirms an explicit gap between the exponents $\xi_{3}(1,1)$ and $\xi_{3}(1,2)$.
\begin{lemma}\label{lem:quenched}
Almost surely in $W$,
\begin{equation}\label{eq:quenched}
 Q_R(W)\le(1-1/1800)^{L_N}.
\end{equation}
Consequently,
\begin{equation}\label{eq:moment-ineq}
 \E Q_R^2\le R^{-\delta_0}\E Q_R+R^{-b},
\end{equation}
where $\delta_0$ is defined in \eqref{eq:delta}.
\end{lemma}

\begin{proof}
We first fix the entire path $W$, so the sets $K_j$ and the good indices of layers are deterministic. By Lemma~\ref{lem:hitting}, if the $j$-th layer is good, then, uniformly in $|x|=r_j$, we have
\[
 \Pp_x(H_{K_j}<\tau_{B(0,r_{j+1})})
 \ge\frac{\capn(K_j)}{9r_j}\ge\frac1{1800}.
\]
Let $A_j$ be the event that the test segment $X[\tau_{r_j},\tau_{r_{j+1}}]$ avoids $K_j$. By the strong Markov property of $X$, we have
\[
 \Pp(A_j\mid\mathcal F^X_{\tau_{r_j}},W)
 \le(1-1/1800)^{\ind_{G_j}},
\]
where $\mathcal F^X$ is the usual Brownian filtration of $X$. For $i<j$, $A_i$ is measurable with respect to $\mathcal F_{\tau_{r_j}}^{X}$. Note that, the event that $X$ avoids $W$ will imply every $A_j$. Hence, by strong Markov property, we have
\[
 Q_R(W)\leq
 \Pp\left(\bigcap_{j=0}^{N-1}A_j\,\middle|\,W\right)
 \le\prod_{j=0}^{N-1}(1-1/1800)^{\ind_{G_j}}
 =(1-1/1800)^{L_N},
\]
which proves~\eqref{eq:quenched}.

Then, applying \eqref{eq:quenched} and noting that $(1-1/1800)^{N/4}=R^{-\delta_0}$, we have
\[Q_R^2=Q_R^2\ind\{L_N\ge N/4\}+Q_R^2\ind\{L_N<N/4\}\leq R^{-\delta_0}Q_R+\ind\{L_N<N/4\}\]
Taking expectations on both sides and using Lemma~\ref{lem:count} proves~\eqref{eq:moment-ineq}.
\end{proof}

We are now ready to prove Theorem~\ref{thm:main}.

\begin{proof}[Proof of Theorem~\ref{thm:main}]
By the second moment estimate in~\eqref{eq:moments}, there exists $c_*>0$ such that $\E Q_R^2\ge c_*R^{-1}$ for all sufficiently large $R$. Applying~\eqref{eq:moment-ineq} along $R=4^N$ gives
\[
 R^{-\delta_0}\E Q_R\ge c_*R^{-1}-R^{-b}.
\]
Since $b>1$, the second term is at most $(c_*/2)R^{-1}$ for all sufficiently large $N$. Thus
\[
 \E Q_R\ge\frac{c_*}{2}R^{-1+\delta_0}.
\]
The first relation in~\eqref{eq:moments} now implies $\xi\le1-\delta_0$. The discrete comparison~\eqref{eq:cut} yields $\beta\ge1+\delta_0$.

Finally $-\log(1-u)>u$ for $0<u<1$, so
\[
 \delta_0>\frac1{7200\log4}>\frac1{10000},
\]
which gives the $1.0001$ bound.
% For a purely rational verification of the last inequality, the positive Taylor series gives
% \[
%  e^{25/36}>\sum_{k=0}^{4}\frac{(25/36)^k}{k!}
%  =\frac{80665009}{40310784}>2.
% \]
% Therefore $\log4<25/18$, as required.
\end{proof}

% The following checks identify the delicate points of the argument. Capacity and the killed Green kernel use the same normalization, so no factor of $2\pi$ is missing. The hitting lemma treats points inside $K$ separately and uses the equilibrium-potential identity only outside $K$. The auxiliary occupation interval may extend past an exit time, but it is used as a measure on $K(r)$ only on the no-exit event. Good-scale events are adapted by~\eqref{eq:adapted}; independence is unnecessary. The test-path windows concatenate exactly, and their bounds are applied without conditioning on future avoidance. Finally, the second moment is a packet event, not pairwise nonintersection of three Brownian paths.

% The arithmetic constants are approximately
% \[
%  p_0\approx0.053700,
%  \qquad b\approx1.063722,
%  \qquad \delta_0\approx0.000100215.
% \]
% These decimals are explanatory only; the proof uses the exact expressions and inequalities above. Subject to the explicitly cited established results, the proof is complete. The value $1.0001$ supplies an explicit constant in the classical strict inequality $\beta>1$; the argument does not establish $1.01$ or $3/2$. In fact, the known inequality $\xi_3(1,1)>1/2$ means that the cut-time exponent $2-\xi_3(1,1)$ is strictly below $3/2$. Proving a lower bound of $3/2$ requires additional information beyond this cut-time comparison.

\end{document}